\documentclass[12pt, a4paper]{amsart}

\usepackage{amsmath,amssymb,amscd,amsfonts}
\newtheorem{theorem}{Theorem}[section]
\newtheorem{lemma}[theorem]{Lemma}
\newtheorem{proposition}[theorem]{Proposition}
\newtheorem{conjecture}[theorem]{Conjecture}

\theoremstyle{definition}     
\newtheorem{definition}[theorem]{Definition}

\newtheorem{example}[theorem]{Example}

\newtheorem{claim}[theorem]{Claim}

\theoremstyle{remark}

\newcommand\Pic{\text{\rm Pic}}

\newcommand\ot{{\otimes}}
\newcommand\OO{{\mathcal{O}}}
\newcommand\PP{{\mathcal{P}}}

\newcommand\II{{\mathcal{I}}}

\newcommand\sh{{\hat {\mathcal S}}}
\newcommand\s{{ {\mathcal S}}}
\newcommand\dimm{{\text{\rm dim}}}
\newcommand\ann{{\text{\rm Ann}}}
\newcommand\rank{{\text{\rm rank}}}

\newcommand\rk{{\text{\rm rank}}}

\newcommand\f[1]{{{\mathfrak{{#1}}}}}

\newcommand\two[1]{{{t_1^{(#1)}}}}

\newcommand\twg[1]{{{t_g^{(#1)}}}}

\begin{document}
\title{On Ueno's conjecture K}
\author{Jungkai A. Chen \and Christopher D. Hacon}
\address{Department of Mathematics, National Taiwan University, Taipei 10617,
Taiwan.} \address{National Center for Theoretical Sciences, Taipei Office.}
\address{Taida Institute for Mathematical Sciences}
 \email{jkchen@math.ntu.edu.tw}

\address{Department of Mathematics,  University of Utah,
155 South 1400 East, JWB 233, Salt Lake City, UT 84112-0090, USA}
\email{hacon@math.utah.edu}
\thanks{The first author was partially supported by NSC, TIMS and NCTS of Taiwan. \\The second author was partially supported by NSF research grant  no: 0456363 and an AMS Centennial Scholarship. We would like to thank J. Koll\'ar, R. Lazarsfeld, C.-H. Liu, M. Popa, P. Roberts, and A. Singh for many useful comments on the contents of this paper.}
\maketitle
\begin{abstract} We show that if
$X$ is a smooth complex projective variety with Kodaira dimension $0$ then
the Kodaira dimension
of a general fiber of its Albanese map is at most $h^0(\Omega ^1 _X)$.
\end{abstract}\section{Introduction}
Let $X$ be a smooth projective
variety with $\kappa(X)=0$. Let $a: X \to A$ be its
Albanese map with general fiber $F$. Then Ueno's Conjecture K states that:
\begin{conjecture} 1) $a$ is an algebraic fiber space (i.e. it is surjective with connected
fibers),

2) $\kappa (F)=0$ and

3) there is an \'etale cover $B\to A$ such that $X\times _A B$ is birationally equivalent to $F\times B$ over $A$.
\end{conjecture}
This conjecture is an important test case of the more general
$C_{n,m}$ conjecture of Iitaka which states that: Given a surjective
morphism of smooth complex projective varieties $f:X\to Y$,
$n=\dim X$ and $m=\dim Y$, with connected general fiber $F$, then
$$\kappa (X)\geq \kappa (F)+\kappa (Y).$$
Kawamata has shown (cf. \cite{Kawamata85}) that these conjectures follow
from the conjectures of the Minimal Model Program (including abundance).
He has also shown:
\begin{theorem}\label{TK1} Conjecture K 1) is true (see Theorem 1 of
\cite{Kawamata81}).\end{theorem}
\begin{theorem}\label{TK2} \cite{Kawamata82} Let $X$ be a smooth projective variety with $\kappa(X)=0$
and $a: X \to A$ its Albanese map with general fiber $F$. If $\dim A =1$,
then $\kappa (F)=0$. \end{theorem}
$C_{n,m}$ then follows easily for any fiber space $f:X\to Y$ where $Y$ is an elliptic curve.
Since it is known that the $C_{n,m}$ conjecture holds when the base is of general type (cf. \cite{Kawamata81} and \cite{Viehweg82}), then $C_{n,1}$ holds.
It is also worth noting that by \cite{kollar87}, conjecture $C_{n,m}$
is known to hold when the general fiber $F$ is of general type.

In this paper we prove the following:
\begin{theorem}\label{mainthm} Let $X$ be a smooth projective variety with $\kappa(X)=0$
and $a: X \to A$ be its Albanese map with general fiber $F$.
Then $\kappa (F)\leq \dim A$.
\end{theorem}

We now proceed to briefly sketch the proof of \eqref{mainthm}.
This loosely follows the main ideas that Kawamata uses in the
proof of \eqref{TK2}. 

In \cite{Hacon04}, using the theory of Fourier-Mukai
transforms, it is shown that if $N\geq 2$ and $P_N(X)=1$, 
then there exists an ideal sheaf $\II
_{N-1} \subset \OO _X$ such that $V_N:=a_*(\omega _X^N\ot \II
_{N-1})$ is a unipotent vector bundle (i.e. given by successive extensions by $\OO _A$) with $h^0(V_N)=1$ and $\rk
(V_N)=P_N(F)$.

We fix an integer $N_1$ such that $|N_1K_F|$ induces the Iitaka
fibration of $F$.
Let $U_t$ be the image of $V_{N_1}^{\ot t}$
in $V_{tN_1}$ (under the natural multiplication map).
Then  $U_t$ is also a unipotent vector bundle and its rank (as a function of $t$) grows like
$t^{\kappa (F)}$. In order to bound the rate of growth of the
ranks of $U_t$, using the theory of Fourier-Mukai
transforms, we consider an equivalent problem concerning modules
$M_t:=R\sh (U_t)$ over the regular local ring $\OO _{\hat{A},\hat{0}}$ of length equal to the rank of $U_t$. Since
$h^0(U_N)=1$ and the multiplication maps above are non-zero on
simple tensors, it turns out that the modules $M_t$ have no
decomposable submodules, and for any submodules $L\subset M_t$ and
$L'\subset M_s$, the dimension of the image of the Pontryagin
product $L*L'$ under the natural map from $M_t*M_s$ to $M_{t+s}$,
always has length at least $\dim _k (L)+\dim _k (L') -1$. This
allows us to define natural extensions $M_t\hookrightarrow
\bar{M}_t$ and multiplication maps $\bar{M}_t*\bar{M}_s \to
\bar{M}_{t+s}$ which behave analogously to the case in which $A$
is a $g$-fold product of elliptic curves. We are hence able to
show that the rate of growth of $\dim _k(\bar{M}_t)$ (and hence of
$\dim _k(M_t)$) is bounded by $O(t^g)$. Therefore $\kappa (F)\leq
g$ where $g=\dim A$.

We believe that a more detailed analysis of the modules $M_t$ would show that the image of the relative Iitaka fibration, denoted by $f:X\to W$,
of the Albanese map $X \to A$ is a
$\mathbb{P}^{\kappa(F)}$-bundle over $A$. We do not pursue this here.

\subsection{Notation and conventions.}\label{NC}
Throughout this paper, we work over $k=\mathbb{C}$. If $X$ is a
smooth projective variety, then $K_X$ will denote a canonical
divisor and we let $\omega _X=\OO _X(K_X)$. For any integer $N>0$,
we let $P_N(X)=h^0(X,\omega _X^N)$ so that $p_g(X)=P_1(X)$. The Kodaira dimension of
$\kappa (X)\in \{ -\infty, 0, 1, \ldots, \dim X \}$ is defined as
follows: if $P_N(X)=0$ for all $N>0$, then $\kappa(X)=-\infty$,
otherwise we let $\kappa (X)=m$, where $0\leq m\leq \dim X$ is an
integer such that $P_N(X)=O(N^m)$ (that is, there are constants
$\alpha,\beta>0$ such that $\alpha N^m\leq P_N(X)\leq \beta N^m$
for all $N\gg 0$ sufficiently divisible).
\subsection{The Fourier-Mukai Transform}
For any abelian variety $A$ of dimension $g$, we let $\hat{A}=\Pic
^0(A)$ be the dual abelian variety and $\PP$ the (normalized)
Poincar\'e line bundle on $A\times \hat{A}$. By \cite{Mukai},
there is a functor $\sh$ from the category of $\OO _A$-modules to
the category of $\OO _{\hat{A}}$-modules defined by $\sh
(M)=(p_{\hat{A}})_* (p_A^*M\ot \PP)$ where $p_A, p_{\hat{A}}$
denote the projections from $A\times \hat{A}$ to $A,\hat{A}$.
Similarly, for any $\OO _{\hat{A}}$-module $N$, one defines $\s
(N)=(p_{{A}})_* (p_{\hat {A}}^*N\ot \PP)$. By (2.2) of
\cite{Mukai}, there are isomorphisms of derived functors $$R\sh \circ R\s
\cong (-1_{\hat {A}})^*[-g],\qquad R\s \circ R\sh \cong (-1_{
{A}})^*[-g],$$ where $R\sh , R\s$ are the the derived functors of
$\sh, \s$ and $[-g]$ denotes shifting a complex $g$ spaces to the
right.

We will consider the dualizing functor on $A$ given by $$\Delta
_A(? )= R{\mathcal H}{\rm om}(?, \OO _A)[g].$$ By 3.8 of
\cite{Mukai}, we have $$\Delta _A\circ R \s=((-1_A)^*\circ R\s
\circ \Delta _{\hat{A}})[g].$$ A vector bundle $U$ on $A$ is {\it
unipotent} if there is a sequence of sub-bundles
$$0=U_0\subset U_1\subset \ldots \subset U_r=U$$ such that
$U_i/U_{i-1}\cong \OO _A$ for all $1\leq i\leq r=\rank (U)$. We
have that $R^g\sh$ gives an equivalence in categories between
unipotent vector bundles on $A$ and the category of coherent
sheaves supported on $\hat {0}$, i.e. the category of Artinian
$\OO _{\hat{A},\hat{0}}$ modules of finite length.

Recall that we have (cf. \S 3 of \cite{Mukai})
\begin{proposition}\label{fm}
Let $U,V$ be unipotent vector bundles on $A$, then $U\otimes V$
and $U^*$ are unipotent vector bundles and we have $R^g\sh
(U\otimes V)\cong R^g\sh (U) * R^g\sh (V)$ and $R^g\sh (U^*)\cong
(-1_{\hat{A}})^*\Delta_{\hat{A}} ( R^g\sh (U))$.
\end{proposition}

Here the Pontryagin product $R^g\sh (U) * R^g\sh (V)$ is just the
push-forward $\mu _*(R^g\sh (U) \boxtimes R^g\sh (V))$ via the
multiplication map $\mu:\hat{A}\times \hat{A}\to \hat{A}$ given by
$\mu(x,y)=x+y$. Equivalently, we regard the vector space $R^g\sh
(U) \otimes _k R^g\sh (V)$ as a $\OO _{\hat{A},\hat{0}}$-module
via the corresponding comultiplication map $\mu ^*:\OO
_{\hat{A},\hat{0}}\to \OO _{\hat{A}, \hat{0}}\otimes \OO
_{\hat{A},\hat{0}}$ given by $\mu ^*(x)=x\otimes 1+1\otimes x$.

\begin{proposition}\label{UtoU}
Let $\phi: U \to V$ be a homomorphism between unipotent vector
bundles. Then $im(\phi), ker(\phi), coker(\phi)$ are also
unipotent.
\end{proposition}
\begin{proof}
We proceed by induction on $r:=\rk (U)$. If $r=1$ the assertion is
clear. If $r\geq 2$, we consider the composition $\OO_A
\stackrel{\imath}{\to} U \stackrel{\phi}{\to} V$.

If $\phi \circ \imath =0$, then we have an induced homomorphism
${\phi}': U/\OO_A \to V$. Therefore, by induction,
$im(\phi)=im({\phi}')$ and $coker(\phi)=coker({\phi}')$ are
unipotent. Moreover we have a short exact sequence $0 \to \OO_A
\to ker(\phi) \to ker({\phi}') \to 0$. Since $ker(\phi)$ is an
extension of unipotent vector bundles, it is also unipotent.

If $\phi \circ \imath  \ne0$, then we have a homomorphism
${\phi}': U/ \OO_A \to V/\OO_A$ such that $ker(\phi) \cong
ker({\phi}')$, $coker(\phi) \cong coker({\phi}')$ . Finally, since
$im(\phi)$ is an extension of $im({\phi}')$ by $\OO_A$, it is also
unipotent.
\end{proof}
\section{Artinian modules without decomposable submodules}
Let $\hat{A}$ be an abelian variety of dimension $g$ with origin
$\hat{0}$ and $B$ be the regular local ring $\OO _{{A},{0}}\cong
\OO _{\hat{A},\hat{0}}$ with maximal ideal $\f m$.

\begin{lemma} The set of all unipotent vector bundles $U$ on $A$ with
$h^0( A,U)=1$ is in one to one correspondence with the set of all
Artinian $B$-modules of finite length without decomposable
submodules. \end{lemma}
\begin{proof} By \cite[Example 2.9]{Mukai}, the Fourier-Mukai transform gives a
bijection between the category of unipotent vector bundles $U$ on $A$ and the category of Artinian
$B$-modules of finite length $M$, given by $U\to M=R^g \sh (U)$
and $M\to U= R^0\s (M)$. Suppose that $M$ has a decomposable
submodule, then there is a injective homomorphism of $B$-modules
$k\oplus k\to M$. Taking the Fourier-Mukai transform we get an
injective homomorphism $\OO _A\oplus \OO _A\to R^0\s (M)$ and
hence $h^0(R^0\s (M))\geq 2$. Suppose on the other hand that
$h^0(U)\geq 2$, then there is a homomorphism $\phi: \OO _A\oplus
\OO _A\to U$ which is injective on global sections. Let $V$ be the
image of $\phi$, then by \eqref{UtoU}, $V$ is a unipotent vector
bundle of rank at most $2$. If $\rk(V)=1$ then $U$ is a unipotent
line bundle with $h^0(V)=2$ which is impossible. Therefore
$\rk(V)=2$ and so $\phi$ is an injection. By \eqref{UtoU}, $coker
(\varphi )$ is also unipotent. In particular $R\sh (coker
(\varphi ))=R^g\sh (coker
(\varphi ))$. Therefore, taking the Fourier-Mukai
transform we get an injective homomorphism $\psi:k\oplus k\to R^g
\sh (U)$ and so $R^g \sh (U)$ has a decomposable submodule.
\end{proof}

We will now define a natural extension of Artinian $B$-modules of finite length with no
decomposable submodules. We begin by considering the
corresponding dual objects.
\begin{lemma} Let $M$ be an Artinian $B$-module of finite length with no
decomposable submodules, then $\Delta (M)$ is an Artinian
$B$-module of finite length with no decomposable quotients.
\end{lemma}
\begin{proof} Clear.
\end{proof}

\begin{lemma}\label{principal}
Let $(B,\f m, k)$ be a local ring with an inclusion $k
\hookrightarrow B$ and $M$ be an Artinian $B$-module of finite length without decomposable
quotient modules. Then $M=Ba$ for some $a \in M$. In particular,
$M \cong B/ \ann (a)$.
\end{lemma}
\begin{proof}
$M/\f m M$ is a  vector space over $k$. If $\dimm_k M/\f m M=0$,
then by Nakayama's Lemma we have $M=0$. If $\dimm_k M/\f m M=1$,
pick any element $a \in M -\f
m M$, then $Ba+\f m M = M$. By Nakayama's Lemma again, $Ba=M$, so that $M
\cong B/ \ann (a)$.

If $\dimm _k M/\f m M\ge 2$, we pick
$a,b \in M- \f m M$ such that their image $\bar{a},\bar{b} \in M/
\f m M$ are linearly independent. Then $B
\bar{a}+B \bar{b}$ is a decomposable quotient module of $M$, which is the
required contradiction.
\end{proof}

Consider now an Artinian $B$-module of finite length $M$ without
decomposable submodules, where $B$ is the local ring
$\OO_{\hat{A},\hat{0}}$. Then $N=\Delta (M)$ is an Artinian
$B$-module of finite length without decomposable quotient modules.
Therefore $N \cong B/I$ for some $I$ with $\sqrt{I} =\mathfrak{m}$.

For any $x \in \f m - \f m ^2$, we let $ e(x):= \min \{k | x^k \in
I\}$. We pick successive elements $$x_i \in \f m  - ({\rm
Span}({x}_1,\ldots,{ x}_{i-1})+\f m ^2)$$ with minimal
$e_i=e(x_i)$. We then let $J=J(I):=(x_1^{e(x_1)},\ldots
,x_g^{e(x_g)}) \subset I$. The elements $(x_1,\ldots,x_g)$ form a
system of parameters and generate $\f m$. We let $\bar{N}:= B/J$
and $\bar{M}=\Delta (\bar{N})$. From the surjection $\bar{N}\to
N$, we obtain an injection $M\to \bar{M}$.

We say that the module $\bar{M}$ (or more precisely the injection $M\to \bar{M}$) is an {\it
algebraically splitting extension} (or {\it AS-extension}), of
$M$.

We will say that a $B$-module of
the form $B/(x_1^{e_1},\ldots ,x_g^{e_g})$ is an {\it AS-module}.
Note that by Lemma \ref{ASD} below, the AS-extensions defined above are AS-modules.
\begin{lemma}\label{ASD} Let $M$ be any AS-B-module then $M\cong \Delta (M)$.
\end{lemma}
\begin{proof}

The proof follows easily by considering the Koszul complex given by the
regular sequence $x_1^{e_1},\ldots , x_g ^{e_g}$.
\end{proof}

In summary we have:
\begin{proposition}
Let $M$ be a Artinian $B$-module without decomposable submodules,
then $M \cong I/J$ for some ideals $I$ and $J=(x_1^{e_1},\ldots ,x_g^{e_g})$.
\end{proposition}

\begin{example}\label{ex1}
Let $M=B/(x_1^{e_1},\ldots ,x_g^{e_g})$ and $M'=B/(x_1^{f_1},\ldots ,x_g^{f_g})$. We wish to compute $M*M'$.
\end{example}

By the description following Proposition \ref{fm}, it can be
realized as  $${M} \ot_k {M'} = k[x_1,\ldots ,x_g] /
(x_1^{e_1},\ldots ,x_g^{e_g}) \ot k[y_1,\ldots ,y_g] /
(y_1^{f_1},\ldots ,y_g^{f_g}),$$ regarded  as $k[t_1,\ldots
,t_g]$-module by letting $t_i:=x_i+y_i$. We first treat the case
that $g=1$. The general case is obtained by taking the tensor product.

Let $e=e_1$, $f=f_1$, $x=x_1$, $y=y_1$ and $t=t_1$. Assume that $e\leq f$.
We will need the following two facts:\\
(1) Let $0\leq d\leq e-1$. Then $t^{e+f-2d-2} v_d \ne 0$ for any
$v_d  \ne 0$  homogeneous of
degree $d$. \\
(2) Let  $0\leq d\leq e-1$. Then there exists an
homogeneous element $t^{(d)}$ of degree $d$, which is unique up to multiplication by a scalar
and such that
$t^{e+f-2d-1} t^{(d)}=0$.

To see these two facts, we write $v_d = \sum_{i=0}^{d} a_i x^i
y^{d-i}$ and $t^{e+f-2d-2} v_d = \sum_{i=1}^{d+1} c_i x^{e-i}
y^{f-d+i-2}$. Then we have

$$ \left( \begin{array}{c} c_1\\ c_2\\ \vdots \\ c_{d+1} \end{array} \right) =
\left( \begin{array}{ccc} C^{e+f-2d-2}_{e-d-1} & \ldots &
C^{e+f-2d-2}_{e-1} \\ \vdots & & \vdots \\
C^{e+f-2d-2}_{e-2d-1} & \ldots & C^{e+f-2d-2}_{e-d-1}
\end{array} \right) \left( \begin{array}{c} a_d\\ a_{d-1}\\ \vdots
\\ a_0
\end{array} \right)
$$
where $C^j_i=j\cdots (j-i+1)/i!$.
An explicit computation (cf. \cite{Roberts})
shows that the $(d+1) \times (d+1)$
matrix is non-singular. This proves $(1)$.

Similarly, if we now write $t^{(d)} = \sum_{i=0}^{d} a_i x^i y^{d-i}$ and
$t^{e+f-2d-1} t^{(d)} = \sum_{i=1}^{d} c_i x^{e-i}
y^{f-d+i-1}$. Then we have

$$ \left( \begin{array}{c} c_1\\ c_2\\ \vdots \\ c_{d} \end{array} \right) =
\left( \begin{array}{ccc} C^{e+f-2d-1}_{e-d-1} & \ldots &
C^{e+f-2d-1}_{e-1} \\ \vdots & & \vdots \\
C^{e+f-2d-1}_{e-2d} & \ldots & C^{e+f-2d-1}_{e-d}
\end{array} \right) \left( \begin{array}{c} a_d\\ a_{d-1}\\ \vdots
\\ a_0
\end{array} \right)
$$
This $d \times (d+1)$ matrix has rank $d$ by an analogous
computation. Thus $(2)$ follows.

We next claim that, as a $k[t]$-module
$$ k[x,y]/(x^e,y^f) = \bigoplus_{d=0}^{e-1} t^{(d)} k[t]/(t^{e+f-2d-1})
.$$ To see this, it suffices to verify the equality as $k$-vector
spaces. Let $F_d$ be the subspace of homogeneous polynomial of
degree $d$. Clearly $F_d$ contains $\{t^{(0)} t^{d}, t^{(1)}
t^{d-1},\ldots , t^{(d-1)} t, t^{(d)}\}$. It suffices to show that
these are linearly independent.  We proceed by induction on $d\in \{0,1,\ldots , e-1\}$.
Hence we assume that $\{t^{(0)} t^{d-1}, t^{(1)} t^{d-2},\ldots ,
t^{(d-2)} t, t^{(d-1)}\}$ are independent, therefore so are $\{t^{(0)}
t^{d}, t^{(1)} t^{d-1},\ldots , t^{(d-1)}t\}$. If $t^{(d)}=t v_{d-1}$
for some $v_{d-1}$, then $ t^{e+f-2d-1} t^{(d)} = t^{e+f-2d}
v_{d-1}=0$ which contradicts $(1)$. Therefore, $t^{(d)} \ne t
v_{d-1}$ and so $t^{(d)}$ is not contained in the subspace
generated by $\{t^{(0)} t^{d},
t^{(1)} t^{d-1},\ldots , t^{(d-1)}t\}$. For $d\in \{ e, e+f-2\}$ the
proof follows by a similar argument. This completes the proof of
the claim.

Therefore, for any $g>0$, we have a decomposition $${M} \ot
{M'}=\bigoplus_{I\in [0,\epsilon _1-1]\times \cdots \times
[0,\epsilon_g-1]} V_{I}$$ where $\epsilon _i={\rm min}\{e_i,f_i \}$ ,
$I=(i_1,\ldots ,i_g)$ and $V_{I}$ is generated by $\two {i_1} \cdots
\twg {i_g}$ and $V_I\cong B/(t_1^{e_1+f_1-1-2i_1},\ldots
,t_g^{e_g+f_g-1-2i_g})$.

We  let $V_{max}:=V_{(0,\ldots,0)}$ be the {\it maximal component
of} ${M} * {M'}$.

\begin{definition}\label{d-g}
Given a multiplication map $\varphi: M*M' \to M_2$ of Artinian $B$-modules of finite length. We say that it
is {\it geometric} if $M_2$ has no decomposable submodules and for
every submodules $L <M$ and $L' <M'$, the image of
$\varphi(L*L') $ has dimension
at least $ \dimm_k(L)+\dimm_k(L ')-1$.
\end{definition}

\begin{proposition}\label{p-as}
Given a geometric multiplication map $\varphi: M*M' \to M_2$.
If $M$ and $M'$ are AS-modules, then restriction of $\varphi$ to the maximal
component $V_{max}$ of $M*M'$ is injective. Moreover,
$\varphi(M*M')=\varphi(V_{max})$.
\end{proposition}

\begin{proof}
Let $M=B/(x_1^{e_1},\ldots ,x_g^{e_g})$ and $M'=B/(x_1^{f_1},\ldots ,x_g^{f_g})$. Then ${M} *{M'}= \bigoplus
V_{I}$ as above.

We first claim that $\varphi: V_{max } \to M_2$ is injective. To see
this, let $N=(x_1^{e_1-1}\cdots
x_g^{e_g-1})/(x_1^{e_1},\ldots,x_g^{e_g})= {\rm Soc} (M)$ and
$N'=(x_1^{f_1-1}\cdots x_g^{f_g-1})/$ $(x_1^{f_1},\ldots,x_g^{f_g})=
{\rm Soc} (M')$ be the unique rank $1$ submodules given by the annihilator of
the maximal ideal. Then $N*N'$ is
the unique rank $1$ submodule in $V_{max}$. Since $\varphi$ is
geometric, $N*N'$ is not in the kernel of $\varphi$. Suppose now
that $0\ne f \in V_{max}$ is in the kernel of $\varphi$, then  there
exist integers $a_1,\ldots ,a_g\geq 0$ such that $t_1^{a_1}\cdots
t_g^{a_g}f$ generates $N*N'$ and this leads to an easy
contradiction.

We will now show that $\varphi(M*M')=\varphi(V_{max})$. Let
$M_2'$ be the image of $\varphi$, then clearly the restriction
$\varphi ':M*M'\to M_2'$ is also geometric. By Proposition \ref{ext} below,
$M_2 '$ admits an extension to $V_{max}$ and hence they are isomorphic
and the claim follows.
\end{proof}

\begin{proposition}\label{ext}
Given a surjective geometric multiplication map $\varphi: M*M' \to
M_2$. Let  $\bar{M}$ and $\bar{M}'$ be a AS-extensions of $M$ and $M'$.
Then $M_2$ admits an extension to the
maximal component of $\bar{M}*\bar{M'}$.
\end{proposition}

\begin{proof}
 We write $\bar{M}=B/(x_1^{e_1},\ldots,x_g^{e_g})$ and $\bar{M}'
=B/(x_1^{f_1},\ldots,x_g^{f_g})$.
We keep the notation as in Example \ref{ex1}.

Let $J=(t_1^{e_1+f_1-1},\ldots ,t_g^{e_g+f_g-1})$.  We note that $J$
annihilates $\bar{M}*\bar{M'}$, and hence it also
annihilates $M*M'$ and $M_2$. We claim that $J$ annihilates $\Delta
(M_2)$. To see this, we regard $t_1,\ldots , t_g$ as elements of
$End_k(M_2)$. Thus $t_i^{e_i+f_i-1}=0 \in End_k(M_2)$ implies that
$t_i^{e_i+f_i-1}=0 \in End_k(\Delta (M_2))$.

Note that $\Delta (M_2)$ is principal and $\Delta (M_2) \cong
B/\ann (\Delta (M_2))$  by Lemma \ref{principal}. Therefore, we have
surjective homomorphism
$$ B/J \twoheadrightarrow \Delta (M_2).$$
Dualizing it, we get an injective homomorphism $M_2
\hookrightarrow\Delta ( B/J) \cong B/J$. Where  $B/J$ is
isomorphic to the maximal component of
$\bar{M}*\bar{M}'$.
\end{proof}
\section{Proof of the main theorem}
In this section we will prove the main theorem.

We will make use of various multiplier ideal sheaves. We refer the
reader to \cite{Lazarsfeld04} for their definitions and main
properties.

\begin{theorem} \label{kdim} Let $a:X\to A$ be a surjective morphism
with general fiber $F$ from
a smooth projective variety $X$ with $\kappa (X)=0$
to an abelian variety with $\dim A=g$.
Then $\kappa (F)\leq g$.
\end{theorem}

\begin{proof}
By \cite[\S 5]{Hacon04}, for all $N\geq 2$ such that $P_N(X)=1$,
there exists an
ideal sheaf $\II _{N-1}\subset \OO _X$ such that
$V_N=a_*(\omega _X^{ N}\otimes \II _{N-1})$ is a unipotent vector bundle
of rank $P_N(F)$ and such that $H^0(A,V_N)=P_N(X)= 1$.
(If $P_N(X)\ne 1$, then $a_*(\omega _X^{ N}\otimes \II _{N-1})$ is given by successive extensions by some $P\in {\rm Pic }_{\rm tors}^0(A)$.)
We let $V_0=\OO _A$. 
(Note that by Theorem 4 of \cite{CH02}, if $P_1(X)=1$, then
$a_*\omega _X =\OO _A$.) We fix an integer $e\geq 2$ such that $P_N(X)=1$
for any integer $N>0$ divisible by $e$.
\begin{lemma}\label{LV}If $N$ and $M$ are positive integers divisible by $e$ then the homomorphisms $V_{N}{\ot}V_M\to a_* \omega _X^{N+M}$ factor through
$V_{N+M}$.
\end{lemma}
\begin{proof}
Let $H$ be an ample line bundle on $A$.
Recall that by definition (cf. \S 5 of \cite{Hacon04}), there exists a number $\epsilon _0>0$ such that
$ \II _{N}=\II (||NK_X+\epsilon a^*H||)$  for any rational number $0<\epsilon \leq \epsilon _0$.
Therefore, we may fix a rational number $1\gg \epsilon >0$ and a sufficiently divisible
integer $t\gg 0$ such that $ \II _{N-1}= \II (\frac 1 t \cdot |t(N-1)K_X+t
\epsilon a^*H|)$, $ \II _{M-1}= \II (\frac 1 t \cdot |t(M-1)K_X+t
\epsilon a^*H|)$ and $ \II _{N+M-1}= \II (\frac 1 t \cdot |t(N+M-1)K_X+t
\epsilon 2 a^*H|)$.
We let $m:A'\to A$ be an \'etale map of abelian varieties such that $\epsilon
m^*H=H'$ is a very ample Cartier divisor.
Let $X'=X\times _A A'$ and let $m':X'\to X$ and
$a':X'\to A'$ be the corresponding morphisms.
Let $V_M'=m^*V_M$, then $V_M'\ot H'$ is generated.
Notice moreover that
\begin{claim} We have $V_N'=(a')_*(\omega _{X'}^N\ot \II _{N-1}')$ where
$\II _{N-1}'=\II (\frac 1 t \cdot |t(N-1)K_{X'}+t (a')^*H'|)$.\end{claim}
\begin{proof}
This follows easily by flat base change and the fact that, by Theorem 11.2.16 of \cite{Lazarsfeld04}, we have
$(m')^*\II (||(N-1)K_X+ \epsilon a^*H||)=\II (||(m')^*((N-1)K_X+
\epsilon a^*H)||)$.
\end{proof}
We now claim that there is a homomorphism
$$(\omega _{X'}^N\ot \II _{N-1}')\ot H^0(X',\OO _{X'}(\omega _{X'}^M\ot \II _{M-1}'\ot (a')^*H'))\to \omega _{X'}^{N+M}\ot \II _{N+M-1}'\ot (a')^*H'.$$
To check this, it suffices to verify that for any section $s\in
H^0(X',\OO _{X'}(\omega _{X'}^M$ $\ot \II _{M-1}'\ot (a')^*H')),$
one has that $\II _{N-1}'\cdot s \subset \omega _{X'}^M \ot (a')^*H'\ot \II _{N+M-1}'$. This in
turn follows from the inclusion of linear series
$$|t((N-1)K_{X'}+(a')^*H')|\times |MK_{X'}+(a')^*H'|^{\times t}\to
|t((N+M-1)K_{X'}+2(a')^*H')|.$$
Pushing forward, we obtain a homomorphism
$$V_N'\otimes H^0(A',V_M'\ot H')\to V_{N+M}'\ot H'\subset (a')_*(\omega _{X'}^{N+M})\ot H'.$$
Since $V_M'\ot H'$ is generated, the map $V_N'\otimes V_M'\ot H'\to (a')_*(\omega _{X'}^{N+M})\ot H' $ factors through $V_{N+M}'\ot H' $.
Therefore, we have a homomorphism $V_N'\otimes V_M'\to V_{N+M}'$ i.e. an element of $$H^0(A',(V_N'\otimes V_M')^*\ot V_{N+M}')\cong \bigoplus _{P\in Ker (\hat{m})}H^0(A,(V_N\otimes V_M)^*\ot V_{N+M}\ot P)$$ $$\qquad \cong H^0(A,(V_N\otimes V_M)^*\ot V_{N+M}).$$
This is the required
homomorphism $V_N\otimes V_M\to V_{N+M}$.
\end{proof}
We now fix a positive integer $N_1>0$ such that
$|N_1K_F|$ defines a rational map which is birational to the
Iitaka fibration of $F$. Let $U_{t}$ be the image of the
homomorphisms $V_{N_1}^{\otimes t}\to V_{tN_1}$. By Proposition
\ref{UtoU} the $U_{t}$ are unipotent vector bundles of rank $r_t$
where $r_t=O(t^{\kappa (F)})$. Let $M_t$ be the Artinian $\OO
_{\hat {A},\hat {0}}$ module given by $M_t=R^g\sh (U_t )$. Then
$\dim _k M_t=r_t$. We have surjective homomorphisms
$$\varphi _{t,s}:M_t *M_s\to M_{t+s}$$
corresponding to the surjective homomorphisms
$U_t\otimes U_s\to U_{t+s}$. We have that:
\begin{lemma} For any submodules $N\subset M_t$ and $N'\subset M_{s}$ one has that
$\dim \varphi _{t,s}(N*N')\geq \dim N+\dim N' -1$.
\end{lemma}
\begin{proof}
It suffices to show that given the sub-bundles $W=R^0\s (N)\subset
U_t$ and $W'=R^0\s (N')\subset U_{s}$, one has that the rank of
the image of $W\otimes W'$ in $U_{t+t'}$ is at least $\rk (W )+\rk
(W' )-1$. This follows easily from the fact that the map
$$\psi:H^0(F,\omega _F^{tN_1})\otimes H^0(F,\omega _F^{sN_1}) \to
H^0(F,\omega _F^{(t+s)N_1})$$ is non-zero on tensors of the form
$0\ne v\otimes w$ and therefore (by a result of H. Hopf) for any
sub-spaces $V$ of $H^0(F,\omega _F^{tN_1})$ and $V'$ of
$H^0(F,\omega _F^{sN_1})$, one has $\dim \psi (V\otimes V') \geq
\dim V +\dim V ' -1$.
\end{proof}
We now define a sequence of finite length modules
$$\bar{M_{t}}\cong B/(x_1^{te_1-t+1},\ldots ,x_g^{te_g-t+1}) $$
and injective homomorphisms
$M_{t}\hookrightarrow \bar{M}_{t}$ as follows:

For $t=1$, $\bar{M}_1$ is the AS-extension of $M_1$ defined as in
\S 2. In particular there is an injective homomorphism
$M_1\hookrightarrow \bar{M}_1$. Assume now that the inclusion
$M_{{t-1}}\hookrightarrow \bar{M}_{{t-1}}$ has been defined. We
consider the geometric multiplication map (in the sense of \S 2)
$$\varphi _{{t-1},1}:M_{{t-1}}*M_{{1}}\to M_{{t}}.$$
Since $\varphi _{{t-1},1}$ is surjective, by Proposition \ref{ext},
as $\bar{M}_{{t}}$ is the maximal component of
$\bar{M}_{{t-1}}*\bar{M}_{{1}}$, then we have an injective
homomorphism $M_{{t}}\hookrightarrow \bar{M}_{{t}}$ and a surjection
$\bar{M}_{{t-1}}*\bar{M}_{{1}}\to \bar{M}_{{t}}$.

Let $e=e_1\cdots e_g$ be the length of
$\bar{M}_1$, then by Example \ref{ex1}, one easily sees that
$\bar{M}_{t}\cong B/(x_1^{te_1-t+1},\ldots ,x_g^{te_g-t+1}) $
as claimed above.
It follows that
$$r_{t}\leq \dim _k\bar{M}_{t}=\prod (t(e_i-1)+1)<e\cdot t^g,$$
and therefore $\kappa (F)\leq g$.
\end{proof}

\end{document}